\documentclass[14pt,reqno]{amsart}
\usepackage[margin=3cm]{geometry}
\usepackage{setspace}
\usepackage{mathrsfs}
\usepackage{amssymb}
\usepackage{amsmath}
\usepackage{color}
\theoremstyle{plain}
\newtheorem{theorem}{Theorem}[section]
\newtheorem{proposition}[theorem]{Proposition}
\newtheorem{lemma}[theorem]{Lemma}

\theoremstyle{definition}
\newtheorem{definition}[theorem]{Definition}

\theoremstyle{remark}
\newtheorem{remark}[theorem]{Remark}

\begin{document}

\title[Global weak solutions for a two-component Novikov system]
{Global weak solutions for a two-component Novikov system}

\author{Zhigang Li}
\address{College of sciences\\China University of Mining and Technology\\100083, Beijing, People's Republic of China}
\email{lzgcumtb@163.com}


\begin{abstract}
In this paper, we mainly consider about the existence and uniqueness of global weak solutions for the two-component Novikov system. We first recall some results and definitions of strong solutions and weak solutions for the system, then by using the method of approximation of smooth solutions, we prove the existence and uniqueness of global weak solutions of the system.

Correspondence should be addressed to Zhigang Li; lzgcumtb@163.com
\end{abstract}

\maketitle

\section{Introduction}
In this paper, we consider the following two-component Novikov system (2NS) :
\begin{equation*}
\begin{cases}
m_t+3u_xvm+uvm_x=0,\\
n_t\ +3v_xun\ +uvn_x\ =0,\\
m=u-u_{xx}, n=v-v_{xx}.
\end{cases}
\end{equation*}
This system firstly appeared in reference [1] by Geng and Xue, who derived by arising the zero curvature equation
$$U_t-V_x+[U,V]=0$$
which equivalent to the compatibility condition of the $3 \times 3$ spectral system
$$\Phi_x=U\Phi,\ \ \ \Phi_t=V\Phi,$$
where the spacial part $U$ is
\begin{equation*}
U=\left(\begin{matrix}
0 & \lambda m & 1\\
0 & 0 & \lambda n\\
1 & 0 & 0
\end{matrix}\right),
\end{equation*}
and the time part $V$ is
\begin{equation*}
V=\left(\begin{matrix}
-u_xv+\frac{1}{3\lambda^2} & -\lambda uvm +\frac{u_x}{\lambda} & u_xv_x\\
-\frac{v}{\lambda} & u_xv-uv_x-\frac{2}{3\lambda^2} & -\lambda uvn-\frac{v_x}{\lambda}\\
-uv & \frac{u}{\lambda} & uv_x+\frac{1}{3\lambda^2}
\end{matrix}\right),
\end{equation*}
here $\lambda \in \mathbb{R}$ is spectral parameter.

The authors proved such system is integrable, and they also supplied many significant results, such as infinite many conserved quantities, Hamiltonian structure, and explicit multi-peakon traveling wave solutions. It should be noted that the bi-Hamiltonian structure was found by Li and Liu[2], which means 2NS is also integrable in Liouville sense. The single peakon solutions of 2NS are given by
$$u_c(t,x)=v_c(t,x)=\sqrt{c}e^{-|x-ct|},$$
and the periodic case are given by
$$u_p(t,x)=v_p(t,x)=\frac{\sqrt{c}}{\cosh \pi}\cosh\left(x-ct-2\pi[\frac{x-ct}{2\pi}]-\pi\right),$$
with periodic $2\pi$. More details on the derivation of peakons can be found in [3].

As 2NS is a multi-component system, we can exploit some reductions to reduce it into some single-component equations. A nature reduction is taking $u=v$, which is just the Novikov equation (NE) [4] with cubic nonlinearities,
\begin{equation*}
m_t+u^2m_x+3uu_xm=0,\ \ \ m=u-u_{xx},
\end{equation*}
it was firstly derived by Novikov via the symmetry classification method. The integrability of NE was shown by Hone and Wang, who proved that it is integrable both in Lax and Liouville sense. The Lax pair of NE can be obtained from one of 2NS by choosing $u=v$. More over, it also admits bi-Hamiltonian structure, infinitely many conserved quantities [5]. They also showed that NE is associated to the negative flow in the Sawada-Kotera hierarchy. The local well-posedness and ill-posedness of such equation can refer to [6-7], while for periodic case and for $s > \frac{5}{2}$, well-posedness had been proved by Tiglay [8], and the blow-up phenomena has studied in [9]. The global weak solutions of NE should be referred to [10].

Another important reduction of 2NS is Degasperis-Procesi equation (DPE) [11] if we take $v=1$,
$$m_t+um_x+3u_xm=0, \ \ \ m=u-u_{xx}.$$
It was proposed by Degasperis and Procesi, who considered the asymptotic integrability to the following dispersive PDE,
\begin{equation*}
u_t- \alpha^2u_{xxt} + \gamma u_{xxx} +c_0u_x =(c_1u^2+c_2u_x^2+c_3uu_{xx})_x.
\end{equation*}
In fact, if we change the coefficient of term $u_xm$ to 2 in DPE, it is just the famous Camassa-Holm equation (CHE) [12-14]. DPE is another integrable equation of b-family for $b=3$. As the same as CHE, DPE also arises bi-Hamiltonian structure, infinite many conserved quantities and peaked solutions, and it is connected with a negative flow in the Kaup-Kupershmidt hierarchy throw reciprocal transformation [15]. The well-posedness and stability of DPE have been shown in [16-17]. The global weak solutions of CHE has been studied in [18].

Himonas and Mantzavinos have illustrated 2NS is local well-posed if $u_0,v_0 \in H^s$ with $s>\frac{3}{2}$ in convolution form [19], which is equivalent to $m_0,n_0 \in H^s$ for $s>\frac{1}{2}$, and we have just proved if the initial data $m_0$ and $n_0$ do not change sign, then the corresponding strong solution $(m,n)$ exists globally in time [20]. While 2NS is ill-posed if $s<\frac{3}{2}$, which has been proven by Himonas, Holliman and Kenig in [21]. Thus, a nature question is how to modify the initial condition to guarantee there exists a unique weak solution globally in time.

The structure of our paper is organized as follows. In Section 2, we recall some basic definitions and lemmas which will be used in the sequel. We study the existence and uniqueness of the global weak solution of 2NS In section 3.

\textbf{Notations:} For $x$ variable, all the function space is over $\mathbb{R}$, and we drop $\mathbb{R}$ in our notations of function spaces if there is no ambiguity. $C_T$ represents a positive constant which depends only on $T$ and some norms of the initial data, and may be changed from line to line.

\section{Strong solutions and some priori estimates}
In this section, we recall some crucial lemmas and global existence of strong solutions to 2NS. For convenience, we also offer some priori estimates in this section. Note that the kernel of $\mathcal{S}^{-2}$ operator $(1-\partial_x^2)^{-1}$ is $g(x)=\frac{1}{2}e^{-|x|}$, which means $(1-\partial_x^2)^{-1}f=g*f$ for any function $f \in L^2$. Thus we can rewrite 2NS into the following convolution form,
\begin{equation}
\begin{cases}
u_t+uu_xv+g(x)*(3uu_xv+u_x^2v_x+uu_{xx}v_x)+\partial_x g(x)*uu_xv_x=0,\\
v_t+uvv_x+g(x)*(3uvv_x+u_xv_x^2+u_xvv_{xx})+\partial_x g(x)*u_xvv_x=0,
\end{cases}
\end{equation}

\begin{definition}
If $u(t,x),v(t,x) \in [C([0,T); H^s)\bigcap C^1([0,T); H^{s-1})]^2$ with $s>\frac{3}{2}$ is a solution to 2NS, then $u(t,x), v(t,x)$ is called a strong solution to 2NS.
\end{definition}

\begin{definition}
Assume $(u_0,v_0) \in H^s \times H^s$ with $s \in[0,\frac{3}{2}]$. If $(u(t,x),v(t,x))$ satisfies
\begin{equation*}
\begin{aligned}
&\int_0^T\int_\mathbb{R}\left(u\varphi_t-[uvu_x+g*(3uvu_x+u_x^2v_x+uu_{xx}v_x)+g_x*uu_xv_x]\varphi\right){\rm d}x{\rm d}t+\int_\mathbb{R}u_0\varphi(t,x){\rm d}x=0,\\
&\int_0^T\int_\mathbb{R}\left(v\varphi_t-[uvv_x+g*(3uvv_x+u_xv_x^2+u_xvv_{xx})+g_x*u_xvv_x]\varphi\right){\rm d}x{\rm d}t+\int_\mathbb{R}v_0\varphi(t,x){\rm d}x=0,\\
\end{aligned}
\end{equation*}
for any test function $\varphi \in C_c^\infty([0,T\times \mathbb{R})$, then $u(t,x), v(t,x)$ is called a weak solution tow 2NS.
\end{definition}

\begin{proposition}$[10]$
~

(1) Every strong solution is a weak solution.

(2) If $(u,v)$ is a weak solution and $u(t,x),v(t,x) \in [C([0,T); H^s)\bigcap C^1([0,T); H^{s-1})]^2$ with $s>\frac{3}{2}$, then it is a strong solution.
\end{proposition}

\begin{lemma}$[18]$
Let $T>0$. If
$$f,g \in L^2((0,T);H^1)\ \ \  and\ \ \ \frac{{\rm d}f}{{\rm d}t},\frac{{\rm d}g}{{\rm d}t} \in L^2((0,T);H^{-1}),$$
then $f$, $g$ are a.e. equal to a function continuous from $[0,T]$ into $L^2$ and
$$<f(t),g(t)>-<f(s),g(s)>=\int_s^t\Big\langle\frac{{\rm d}f(\tau)}{{\rm d}\tau},g(\tau)\Big\rangle {\rm d}\tau+\int_s^t\Big\langle\frac{{\rm d}g(\tau)}{{\rm d}\tau},f(\tau)\Big\rangle {\rm d}\tau$$
for all $s,t \in [0.T]$, where $< , >$ denote the duality paring between $H^1$ and $H^{-1}$.
\end{lemma}

Throughout this paper, let $\{\rho_n\}_{n=1}^\infty$ denote the mollifiers
$$\rho_n(x)=\left(\int_\mathbb{R}\rho(x){\rm d}x\right)^{-1}n\rho(nx),\ \ \ x \in \mathbb{R}, n \in \mathbb{N}^*,$$
where $\rho \in C_c^\infty$ is defined by
\begin{equation*}
\rho(x)=\begin{cases}e^{\frac{1}{x^2-1}},\ \ \ &|x|<1,\\
0,&|x| \ge 1.
\end{cases}
\end{equation*}

Next, we recall some crucial approximations results of convolution calculus.
\begin{lemma}$[18]$
Suppose $f:\mathbb{R} \rightarrow \mathbb{R}$ be uniformly continuous and bounded. If $g\in L^p$ with $1 \le p \le \infty$, then we have
$$[\rho_n *(fg)-(\rho_n*f)(\rho_n*g)]\rightarrow 0,$$
as $n\rightarrow \infty$ in $L^p$.
\end{lemma}

\begin{lemma}$[18]$
Suppose $f(t,\cdot) \in W^{1,1}$ is uniformly bounded in $W^{1,1}$ for all $t \in \mathbb{R}_+$. Then for a.e. $t \in \mathbb{R}_+$ and $1 \le p <\infty$, we have
$$\frac{1}{p}\frac{\rm d}{{\rm d}t}\int_\mathbb{R}|\rho_n*f|^p{\rm d}x=\int_\mathbb{R}|\rho_n*f|^{p-1}(\rho_n*f_t){\rm sgn}(\rho_n*f){\rm d}x,$$
and
$$\frac{1}{p}\frac{\rm d}{{\rm d}t}\int_\mathbb{R}|\rho_n*f_x|^p{\rm d}x=\int_\mathbb{R}|\rho_n*f_x|^{p-1}(\rho_n*f_{xt}){\rm sgn}(\rho_n*f_x){\rm d}x.$$
\end{lemma}

\begin{lemma}$[20]$
Suppose $u_0,v_0 \in H^s$, with $s \ge 3$, $m_0=(1-\partial_x^2)u_0$ and $n_0=(1-\partial_x^2)v_0$ are non-negative, then 2NS has a unique global strong solution
$$(u(t,x),v(t,x)) \in C(\mathbb{R}_+;H^s)\cap C^1(\mathbb{R}_+;H^{s-1}).$$
Moreover, $E(u,v)=\int_\mathbb{R}\Big(u(t,\cdot)v(t,\cdot)+u_x(t,\cdot)v_x(t,\cdot)\Big){\rm d}x$ is a conservation law.
\end{lemma}

At the end of this section, we offer some priori estimates for $u$, $v$ and potentials $m$, $n$.
\begin{lemma}$[20]$
Suppose $u_0,v_0 \in H^3$ and $m_0, n_0 \in L^p$ with $1 \le p \le \infty$ are non-negative, then there exists some constant $C>0$ only depends on the maximal existence time $T>0$ and initial data $(u_0,v_0)$, such that
\begin{equation*}
\begin{aligned}
(i)&\|u_x(t,\cdot)\|_{L^\infty} \le \|u(t,\cdot)\|_{L^\infty} \le \frac{\sqrt{2}}{2}\|u(t,\cdot)\|_{H^1} \le \frac{\sqrt{2}}{2}\|u_0\|_{H^1}e^{E(u_0,v_0)}t,\\
(ii)&\|v_x(t,\cdot)\|_{L^\infty} \le \|v(t,\cdot)\|_{L^\infty} \le \frac{\sqrt{2}}{2}\|v(t,\cdot)\|_{H^1} \le \frac{\sqrt{2}}{2}\|v_0\|_{H^1}e^{E(u_0,v_0)}t,\\
(iii)&\|u(t,\cdot)\|_{L^p},\|u_x(t,\cdot)\|_{L^p} \le \|m(t,\cdot)\|_{L^p} \le e^{Ct}\|m_0\|_{L^p},\\
(iv)&\|v(t,\cdot)\|_{L^p},\|v_x(t,\cdot)\|_{L^p} \le \|n(t,\cdot)\|_{L^p} \le e^{Ct}\|n_0\|_{L^p}.
\end{aligned}
\end{equation*}
\end{lemma}

\begin{proof}
For (i) and (ii), one can follow the proof of lemma 2.7 in [x]. Now we illustrate the third inequality. By virtue of the first equation of 2NS, we have
\begin{equation*}
\begin{aligned}
\frac{\rm d}{{\rm d}t}\int_\mathbb{R}m^p{\rm d}x&=\int_\mathbb{R}m^{p-1}m_t{\rm d}x=-\int_\mathbb{R}m^{p-1}(uvm_x+3u_xmv){\rm d}x=\int_\mathbb{R}(uv_x-2u_xv)m^p{\rm d}x\\
& \le (2\|u_x\|_{L^\infty}\|v\|_{L^\infty}+\|u\|_{L^\infty}\|v_x\|_{L^\infty})\int_\mathbb{R}m^p{\rm d}x\le C_T\int_\mathbb{R}m^p{\rm d}x,
\end{aligned}
\end{equation*}
and take advantage of Young's inequality and Gronwall's inequality, we obtain
\begin{equation*}
\begin{aligned}
\|u(t,\cdot)\|_{L^p}&=\|(g*m)(t,\cdot)\|_{L^p} \ \le \ \|g(t,\cdot)\|_{L^1}\|m(t,\cdot)\|_{L^p}\ \le e^{Ct}\|m_0\|_{L^p},\\
\|u_x(t,\cdot)\|_{L^p}&=\|(g_x*m)(t,\cdot)\|_{L^p} \le \|g_x(t,\cdot)\|_{L^1}\|m(t,\cdot)\|_{L^p} \le e^{Ct}\|m_0\|_{L^p}.
\end{aligned}
\end{equation*}
Similarly, one can check that (iv) is also holds.
\end{proof}

\section{Global weak solution}
Inspired by Zheng and Yin [22], now we state our main results as following theorem.
\begin{theorem}
Suppose $m_0,n_0 \in L^1\cap L^2$ and $m_0,n_0 \ge 0,\forall x\in\mathbb{R}.$ Then 2NS has a unique global weak solution $(u,v) \in C(\mathbb{R}_+;H^1)\cap C^1(\mathbb{R}_+;L^2)$, and the potential $(m,n) \in L_{loc}^\infty(\mathbb{R}_+;L^1\cap L^2)$.
\end{theorem}

\begin{proof}

\textbf{Step 1.} Let $u_0=(1-\partial_x^2)^{-1}m_0$, $v_0=(1-\partial_x^2)^{-1}n_0$. As $m_0,n_0 \in L^2$, by virtue of Lemma 2.8 with $p=2$, we have $u_0,v_0 \in H^1$. Define $u_0^k=\rho_k*u_0$ and $v_0^k=\rho_k*v_0$ belong to $H^\infty$ for $k \in \mathbb{N}^*$. By the convolution approximation, it is clearly that
$$u_0^k \rightarrow u_0\ \ \ {\rm and}\ \ \ v_0^k \rightarrow v_0\ \ \ {\rm in}\  H^1\ \ \ {\rm as}\ \ \ k\rightarrow \infty.$$
Noting that $m_0^k=u_0^k-u_{0,xx}^k=\rho_k*m_0 \ge 0$, and $n_0^k=v_0^k-v_{0,xx}^k=\rho_k*n_0 \ge 0$, for $k\in\mathbb{N}^*$, it is obvious that
$$m_0^k \rightarrow m_0\ \ \ {\rm and}\ \ \ n_0^k \rightarrow n_0\ \ \ {\rm in}\ L^1\cap L^2\ \ \ {\rm as}\ \ \ k\rightarrow \infty.$$
As $m_0$ and $n_0$ are smooth function in $H^\infty$, through Lemma 2.7 we have unique global strong solutions $m^k,n^k$ of 2NS with initial data $m_0^k,n_0^k$, and $u^k=g*m^k$, $v^k=g*n^k$.

\textbf{Step 2.} For fixed $T>0$, by Lemma 2.8 and the fact $\|\rho_k*f\|_{W^{k,p}}\le\|\rho_k\|_{L^1}\|f\|_{W^{k,p}}$, we see $u^k,v^k$ is uniformly bounded in $L^\infty([0,T);H^1\cap W^{1,\infty})$, and $m^k,n^k$ is uniformly bounded in $L^\infty([0,T);L^1\cap L^2)$. In order to obtain the uniform boundedness of solution $u^k,v^k$ in $H^1((0,T)\times\mathbb{R})$ with initial data $u_0^k,v_0^k$, it remains to estimate $\|u_t^k(t,\cdot)\|_{L^2}$ and $\|v_t^k(t,\cdot)\|_{L^2}$. Rewriting the first equation of 2NS, we have
\begin{equation}
u_t^k+g_x*u^kv^km^k+2g*u_x^kv^km^k-g*u^kv_x^km^k=0.
\end{equation}
Taking advantage of Young's inequality yields
\begin{equation}
\begin{aligned}
&\|u_t^k(t,\cdot)\|_{L^2}\\
\le &\|(g_x*u^kv^km^k)(t,\cdot)\|_{L^2}+2\|(g*u_x^kv^km^k)(t,\cdot)\|_{L^2}+\|(g*u^kv_x^km^k)(t,\cdot)\|_{L^2}\\
\le &\|m^k(t,\cdot)\|_{L^2}\left(\|u^k(t,\cdot)\|_{L^\infty}\|v^k(t,\cdot)\|_{L^\infty}+2\|u_x^k(t,\cdot)\|_{L^\infty}\|v^k(t,\cdot)\|_{L^\infty}+\|u^k(t,\cdot)\|_{L^\infty}\|v_x^k(t,\cdot)\|_{L^\infty}\right)\\
\le &2\|m^k(t,\cdot)\|_{L^2}\|u^k(t,\cdot)\|_{W^{1,\infty}}\|v^k(t,\cdot)\|_{W^{1,\infty}},
\end{aligned}
\end{equation}
and the boundary of $\|v_t^k\|_{L^2}$ can be obtained in s similar way. Combining with the fact
$$\|u^k(t,\cdot)\|_{H^1} \le \|u(t,\cdot)\|_{H^1},\ \ \ \|v^k(t,\cdot)\|_{H^1} \le \|v(t,\cdot)\|_{H^1},$$
and Lemma 2.8, we see sequence $\{u^k,v^k\}_{k\in\mathbb{N}^*}$ is uniformly bounded in $H^1((0,T)\times\mathbb{R})$. Therefore, there exists a subsequence such that
\begin{equation}
\begin{aligned}
&(u^{k_j},v^{k_j}) \rightarrow (u,v)\ {\rm weakly\ in}\ H^1((0,T)\times\mathbb{R}),\ {\rm as}\ k_j\rightarrow \infty,\\
&(u^{k_j},v^{k_j}) \rightarrow (u,v)\ {\rm a.e.\ on}\ (0,T)\times\mathbb{R}, \ {\rm as}\ k_j\rightarrow \infty,
\end{aligned}
\end{equation}
for some $u,v\in H^1((0,T)\times\mathbb{R})$. By Lemma 2.7 and Lemma 2.8, we have that for fixed $t \in (0,T)$,
$$\mathbb{V}[u_x^{k_j}(t,\cdot)]=\|u_{xx}^{k_j}(t,\cdot)\|_{L^1}\le \|u^{k_j}(t,\cdot)\|_{L^1}+\|m^{k_j}(t,\cdot)\|_{L^1} \le 2\|m^{k_j}\|_{L^\infty((0,T);L^1)},$$
and
\begin{equation*}
\begin{aligned}
\|u_x^{k_j}(t,\cdot)\|_{L^\infty} &\le \frac{\sqrt{2}}{2}\|u^{k_j}(t,\cdot)\|_{H^1}\le \frac{\sqrt{2}}{2}\|u_0^{k_j}\|_{H^1}e^{E(u_0^{k_j},u_0^{k_j})t}\\
&\le \frac{\sqrt{2}}{2}\|u_0^{k_j}\|_{H^1}e^{\|u_0^{k_j}\|_{H^1}\|v_0^{k_j}\|_{H^1}t} \le \frac{\sqrt{2}}{2}\|u_0\|_{H^1}e^{\|u_0\|_{H^1}\|v_0\|_{H^1}t},
\end{aligned}
\end{equation*}
where $\mathbb{V}(f)$ is the total variation of function $f\in BV$. Applying Helly's theorem and the diagonal process, we obtain that there exists a subsequence, denoted again by $u_x^{k_j}(t, x)$, which converges pointwise for all $x \in \mathbb{R}$ and $t$ in a countable dense subset of $(0, T)$. In order to prove  $u_x^{k_j}(t, x)$ converges for a.e. $x$ and every $t \in (0, T)$, as $k_j \rightarrow \infty$, it will suffice to prove that $\|u_{xt}^{k_j}\|_{L^\infty((0,T);L^1)}$ is uniformly bounded. Differentiating equation (3) with respect to $x$, and together with the identity $\partial_x^2g*f=g*f-f$, we get
$$u_{xt}^k+g*u^kv^km^k-u^kv^km^k+2g_x*u_x^kv^km^k-g_x*u^kv_x^km^k=0.$$
Similar to the argument of (3), it is easy to check that there exists some constant $C_T>0$ such that
$$\|u_{xt}^{k_j}\|_{L^\infty((0,T);L^1)}\le C_T.$$
Thus we claim there exists some function such that $u_x^{k_j}(t,x)\rightarrow r(t,x)$ for a.e. $x\in \mathbb{R}$ and $\forall t\in (0,T)$ as $k_j\rightarrow \infty$. Since for almost all $(t,x)\in(0,T)\times\mathbb{R}$, $u_x^{k_j} \rightarrow u_x$ in $\mathcal{D}'((0,T)\times \mathbb{R})$, it follows that $r(t,x)=u_x(t,x)$ for a.e. $(t,x)\in(0,T)\times\mathbb{R}$. With a similar argument, we can also deduce that $v_x^{k_j} \rightarrow v_x$ a.e. on $(0,T)\times\mathbb{R}$. Therefore, we have
\begin{equation}
(u_x^{k_j}, v_x^{k_j})\rightarrow (u_x,v_x)\ \ \ {\rm a.e.\ on}\ (0,T)\times\mathbb{R}.
\end{equation}
Moreover, $u_x,v_x \in L^\infty((0,T)\times\mathbb{R})$.

Since $(m^k,n^k)$ is uniformly bounded in $L^\infty((0,T);L^1\cap L^2)$, there exists a subsequence such that
\begin{equation}
(m^{k_j},n^{k_j}) \rightarrow (m,n)\ \ \ {\rm weakly\  *\  in}\ L^\infty((0,T);L^1\cap L^2),\ \ \ {\rm as}\ k_j\rightarrow \infty.
\end{equation}

Next we show $m,n$ satisfy 2NS in distribution sense. We only deal with the term $u_xvm$ as an example, and the others are similar. By virtue of (5)-(7) and Lemma 2.5 , we see for any test function $\phi \in C_0^\infty((0,T)\times \mathbb{R})$,
\begin{equation*}
\begin{aligned}
&\int_0^T\int_\mathbb{R}\left(u_x^{k_j}v^{k_j}m^{k_j}-u_xvm\right)\phi{\rm d}x{\rm d}t\\
=&\int_0^T\int_\mathbb{R}m^{k_j}[u_x^{k_j}v^{k_j}-(u_xv)^{k_j}+(u_xv)^{k_j}-u_xv]\phi{\rm d}x{\rm d}t+\int_0^T\int_\mathbb{R}(m^{k_j}-m)u_xv\phi{\rm d}x{\rm d}t\\
\le&\|m^{k_j}\|_{L^\infty((0,T);L^2)}\|[u_x^{k_j}v^{k_j}-(u_xv)^{k_j}]\phi+[(u_xv)^{k_j}-u_xv]\phi\|_{L^1((0,T);L^2)}\\
&+\|m^{k_j}-m\|_{L^\infty((0,T);L^2)}\|u_xv\phi\|_{L^1((0,T);L^2)}\rightarrow 0,\ \ \ {\rm as}\ k_j\rightarrow \infty.
\end{aligned}
\end{equation*}

\textbf{Step 3.} Since for any fixed $T>0$, $(u^k,v^k)$ is bounded in $L^\infty((0,T);H^1)$, and $(u_t^k,v_t^k)$ is bounded in $L^\infty((0,T);L^2$ respectively, hence the map $t\rightarrow (u^k(t\,cdot), v^k(t,\cdot)) \in H^1$  is weakly equicontinuous on $[0,T]$. By virtue of Arzela-Ascoli theorem, we know there exists a subsequence of $(u^k(t,\cdot),v^k(t,\cdot))$ which converges weakly in $H^1$, uniformly in $t \in [0,T]$. The limit function is $(u,v)$. Since $T>0$ is arbitrary, we claim that
$$(u,v)\in C_w(\mathbb{R}_+;H^1).$$

In order to prove $u,v \in C(\mathbb{R}_+; H^1)$, it remains to show that $\|u(t,\cdot)\|_{H^1}$ and $\|v(t,\cdot)\|_{H^1}$ are continuous on $\mathbb{R}_+$. Since $(m,n)$ is a solution of 2NS in distributional sense, we see for a.e. $t\in \mathbb{R}_+$,
$$\rho_k*m_t+\rho_k*(uvm)_x-\rho_k*(uv_xm)+2\rho_k*(u_xvm)=0.$$
Multiplying with $\rho_k*u$ and integrating with respect to $x$ on $\mathbb{R}$, through Lemma 2.4, we have
\begin{equation}
\begin{aligned}
\frac{1}{2}\frac{\rm d}{{\rm d}t}\|\rho_k*u\|_{H^1}^2&=\int_\mathbb{R}\left[(\rho_k*u_x)(\rho_k*uvm)+(\rho_k*u)(\rho_kuv_xm)-2(\rho_k*u)(\rho_k*u_xvm)\right]{\rm d}x\\
&=G_k(t).
\end{aligned}
\end{equation}
Note that $u,v\in L^\infty((0,T);W^{1,\infty})$, $m,n\in L^\infty((0,T);L^1))$. Thus we also have $u,v\in L^\infty((0,T);W^{2,1}\hookrightarrow C_B^1)$, and for $t\in \mathbb{R}_+$, thanks to Lemma 2.5 we have
$$\lim_{k\rightarrow \infty}G_k(t)=\int_\mathbb{R}(u^2v_xm-uu_xvm){\rm d}\tau=G(t).$$
On the other hand, with H\"{o}lder inequality and Young's inequality, it is easy to check $|G_k(t)| \le C_T$ for $t\in (0,T)$, $k \in \mathbb{N}^*$. Thus integrating (8) with respect to $t$ and letting $k\rightarrow \infty$, we arrive at
$$\frac{1}{2}\|u(t,\cdot)\|_{H^1}^2-\frac{1}{2}\|u_0\|_{H^1}^2=\int_0^tG(\tau){\rm d}\tau.$$
This proves $u \in C(\mathbb{R}_+;H^1)$. With a similar argument, we also have $v \in C(\mathbb{R}_+;H^1)$. Because $u,v\in H^1((0,T)\times \mathbb{R})$, by Lemma 2.4 and (3), it is clear that $u,v\in C^1(\mathbb{R}_+;L^2)$.

\textbf{Step 4.} Finally, we prove the weak solution $(m,n)$ is unique. Suppose $(m_i,n_i)$ are two weak solutions of 2NS in $L_{loc}^\infty(\mathbb{R}_+;L^1\cap L^2)$, $i=1,2$, and the corresponding functions $(u_i,v_i)=(g*m_i,g*n_i) \in C(\mathbb{R}_+;H^1)\cap C^1(\mathbb{R}_+;L^2)$. Define the error of two solutions by $M=m_1-m_2$, $N=n_1-n_2$, $U=u_1-u_2$ and $V=v_1-v_2$, and we denote for any fixed $T>0$,
\begin{equation}
L=\sup_{t \in (0,T),i=1,2}\{\|m_i(t,\cdot)\|_{L^1\cap L^2}+\|n_i(t,\cdot)\|_{L^1\cap L^2}\},
\end{equation}
and with Young's inequality, we have
\begin{equation}
\|u_i(t,\cdot)\|_{W^{1,1}\cap W^{1,\infty}} \le \|m_i(t,\cdot)\|_{L^1} \le L,\ \ \ \|v_i(t,\cdot)\|_{W^{1,1}\cap W^{1,\infty}} \le \|n_i(t,\cdot)\|_{L^1} \le L.
\end{equation}

Substituting $u_i$, $v_i$, $m_i$ and $n_i$ into 2NS, it is easy to derive that the error terms $U$ and $U_x$ satisfy the following evolution equations
\begin{equation}
\begin{aligned}
&& U_t&+Uu_{1x}v_1+u_2U_xv_1+u_xu_{2x}V+g*\Big(3Uu_{1x}v_1+3u_2U_xv_1+3u_2u_{2x}V+(u_{1x}+u_{2x})U_xv_{1x}\\
&& &+u_{2x}^2V_x+\underline{U(u_1-m_1)v_{1x}}-u_{2x}U_xv_{1x}+u_2U_x(n_1-v_1)+u_2(u_2-m_2)V_x\Big)\\
&& &+g_x*\Big(Uu_{1x}v_{1x}+2\underline{u_2U_xv_{1x}}+u_xu_{2x}V_x\Big)=0,
\end{aligned}
\end{equation}

\begin{equation}
\begin{aligned}
&& U_{xt}&+U_xu_{1x}v_1+U(u_1-m_1)v_1+u_{2x}U_xv_1+\underline{u_2U_{xx}v_1}-u_2U_xv_{1x}+u_{2x}^2V+\underline{u_2(u_2-m_2)V}+g_x*\\
&& &\Big(3Uu_{1x}v_1+3u_2U_xv_1+3u_2u_{2x}V+(u_{1x}+u_{2x})U_xv_{1x}+u_{2x}^2V_x+U(u_1-m_1)v_{1x}-u_2U_xv_{1x}\\
&& &+u_2U_xv_{1x}+u_2U_x(n_1-v_1)+u_2(u_2-m_2)V_x\Big)+g*\Big(Uu_{1x}v_{1x}+2u_2U_xv_{1x}+u_2u_{2x}V_x\Big)=0.\\
\end{aligned}
\end{equation}

Our goal is to construct the following differential inequality
\begin{equation}
\begin{aligned}
&\frac{\rm d}{{\rm d}t}\|\rho_k*U\|_{L^2}+\frac{\rm d}{{\rm d}t}\|\rho_k*U_x\|_{L^2}+\frac{\rm d}{{\rm d}t}\|\rho_k*V\|_{L^2}+\frac{\rm d}{{\rm d}t}\|\rho_k*V_x\|_{L^2}\\
\le &C\left(\|\rho_k*U\|_{L^2}+\|\rho_k*U_x\|_{L^2}+\|\rho_k*V\|_{L^2}+\|\rho_k*V_x\|_{L^2}\right)+R_k(t),
\end{aligned}
\end{equation}
where
\begin{equation}
R_k(t)\rightarrow 0\ \ \  {\rm as}\ \ \  k\rightarrow \infty, \ \ \ {\rm  and}\ \ \  |R_k(t)| \le C_K,\ \ \  \forall t \in (0,T), k \in \mathbb{N}^*.
\end{equation}

In order to obtain (13), we should convolute (11) and (12) with $\rho_k$ and take inner product with $|\rho_k*U|{\rm sgn}(\rho_k*U)$ and $|\rho_k*U_x|{\rm sgn}(\rho_k*U_x)$ respectively. Due to the estimates are huge and cumbersome, here we only offer details for some special terms which are underlined.

Line 1:
\begin{equation*}
\begin{aligned}
&\int_\mathbb{R}|\rho_k*U|{\rm sgn}(\rho_k*U)\rho_k*g*(Um_1v_{1x}){\rm d}x\\
\le &\|\rho_k*U\|_{L^2}\|\rho_k*g*(Um_1v_{1x})\|_{L^2}\\
\le &\|\rho_k*U\|_{L^2}\|(\rho_k*U)(\rho_k*m_1)(\rho_k*v_{1x})\|_{L^1}+R_k(t)\\
\le &\|\rho_k*U\|_{L^2}\|\rho_k*U\|_{L^2}\|\rho_k*m_1\|_{L^2}\|\rho_k*v_{1x}\|_{L^\infty}+R^1_k(t)\\
\le &C\|\rho_k*U\|_{L^2}^2+R_k(t),
\end{aligned}
\end{equation*}
where
$$R^1_k(t)=\|\rho_k*U\|_{L^2}\left(\|\rho_k*(Um_1v_{1x})\|_{L^1}-\|(\rho_k*U)(\rho_k*m_1)(\rho_k*v_{1x})\|_{L^1}\right).$$
As $m_1 \in L^1$ and $U,v_1 \in W^{1,\infty}$, we see $R^1_k$ satisfies (14) via Lemma 2.5.

Line 2:
\begin{equation*}
\begin{aligned}
&\int_\mathbb{R}|\rho_k*U|{\rm sgn}(\rho_k*U)\rho_k*g_x*(u_2U_xv_{1x}){\rm d}x\\
\le &\|\rho_k*U\|_{L^2}\|\rho_k*(U_xu_2v_{1x})\|_{L^2}\\
=&\|\rho_k*U\|_{L^2}\|(\rho_k*U_x)(\rho_k*u_2v_{1x})\|_{L^2}+R^2_k(t)\\
\le &\|\rho_k*U\|_{L^2}\|\rho_k*U_x\|_{L^2}\|\rho_k*u_2\|_{L^\infty}\|\rho_k*v_{1x}\|_{L^\infty}+R^2_k(t)\\
\le &C\|\rho_k*U\|_{L^2}\|\rho_k*U_x\|_{L^2}+R^2_k(t),
\end{aligned}
\end{equation*}
with
$$R^2_k(t)=\|\rho_k*U\|_{L^2}\left(\|\rho_k*(U_xu_2v_{1x})\|_{L^2}-\|(\rho_k*U_x)(\rho_k*u_2)(\rho_k*v_{1x})\|_{L^2}\right),$$
which is also tends to zero as $k \rightarrow \infty$ by virtue of $U,u_1,v_1 \in H^1$.

Line 3:
\begin{equation*}
\begin{aligned}
&\int_\mathbb{R}|\rho_k*U_x|{\rm sgn}(\rho_k*U_x)\rho_k*(U_{xx}u_2v_1){\rm d}x\\
=&\int_\mathbb{R}|\rho_k*U_x|{\rm sgn}(\rho_k*U_x)(\rho_k*U_{xx}(\rho_k*u_2)(\rho_k*v_1){\rm d}x+R^3_k(t)\\
=&\frac{1}{2}\int_\mathbb{R}\frac{\partial}{\partial x}|\rho_k*U_x|^2(\rho_x*u_2)(\rho_k*v_1){\rm d}x+R^3_k(t)\\
=&-\frac{1}{2}\int_\mathbb{R}|\rho_k*U_x|^2\left[(\rho_k*u_{2x})(\rho_k*v_1)+(\rho_k*u_2)(\rho_k*v_{1x})\right]{\rm d}x+R^3_k(t)\\
\le &\frac{1}{2}\int_\mathbb{R}|\rho_k*U_x|^2\|\rho_k*u_2\|_{W^{1,\infty}}\|\rho_k*v_1\|_{W^{1,\infty}}+R^3_k(t),
\end{aligned}
\end{equation*}
with
$$R^3_k(t)=\int_\mathbb{R}|\rho_k*U_x|{\rm sgn}(\rho_k*U_x)\left[\rho_k*(U_{xx}u_2v_1)-(\rho_k*U_{xx})(\rho_k*u_2)(\rho_k*v_1)\right]{\rm d}x.$$

Line 4:
\begin{equation*}
\begin{aligned}
&\int_\mathbb{R}|\rho_k*U_x|{\rm sgn}(\rho_k*U_x)\rho_k*(u_2m_2V){\rm d}x\\
\le &\|\rho_k*U_x\|_{L^2}\|\rho_k*(u_2m_2V)\|_{L^2}\\
=& \|\rho_k*U_x\|_{L^2}\|(\rho_k*u_2)(\rho_k*m_2)(\rho_k*V)\|_{L^2}+R^4_k(t)\\
\le &\|\rho_k*U_x\|_{L^2}\|\rho_k*u_2\|_{L^\infty}\|\rho_k*m_2\|_{L^2}\|\rho_k*V\|_{L^\infty}+R^4_k(t)\\
\le &C\|\rho_k*U_x\|_{L^2}\|\rho_k*V\|_{H^1}+R^4_k(t),
\end{aligned}
\end{equation*}
with
$$R^4_k(t)=\|\rho_k*U_x\|_{L^2}\left(\|\rho_k*(u_2m_2V)\|_{L^2}-\|(\rho_k*u_2)(\rho_k*m_2)(\rho_k*V)\|_{L^2}\right).$$

For $V$ and $V_x$, one can following analogous estimates as above to check that both $\dfrac{\rm d}{{\rm d}x} \|\rho_k*V\|_{L^2}$ and
$\dfrac{\rm d}{{\rm d}x} \|\rho_k*V_x\|_{L^2}$ can be bounded by the right hand side of (13). Denote
$$A_k(t)=\|\rho_k*U\|_{L^2}+\|\rho_k*U_x\|_{L^2}+\|\rho_k*V\|_{L^2}+\|\rho_k*V_x\|_{L^2},$$
applying Young's inequality to (13) yields

\begin{equation*}
\begin{aligned}
A_k(t)=e^{Ct}\left(A_k(0)+\int_0^te^{-Cs}R_k(s){\rm d}s\right).
\end{aligned}
\end{equation*}
Since $R_k(t)$ satisfies (14) and $A_k(0)=0$, the uniqueness is obtained by letting $k \rightarrow \infty$.

\end{proof}

\begin{remark}
In the proof of Theorem 3.1, in order to obtain (5), we considered the total variation of $u_x^{k_j}$ and $v_x^{k_j}$, this is the reason we have to restrict $m_0,n_0 \in L^1$, and the uniqueness is guaranteed since our assumption is $m_0$ and $n_0$ also belong to $L^2$, and actually $p=2$ is the critical value of $L^p$ ($p \ge 2$) such that we can deduce the uniqueness of global weak solutions. More details can be found in the reference [22].
\end{remark}

\begin{remark}
As we have shown in introduction, 2NS admits peakon and periodic peakon solutions, and it is obvious $m_c(t,x)=n_c(t,x)=(1-\partial_x^2)\sqrt{c}e^{-|x-ct|}$ are not $L^p$ integrable function. Thus, how to weaken the conditions of $m_0$ and $n_0$ is still an unsolved problem.
\end{remark}

\end{document}